\documentclass[a4paper,reqno]{amsart}

\usepackage{amssymb}
\usepackage{latexsym}
\usepackage{amsmath}
\usepackage{amsthm}
\usepackage{euscript}

\usepackage{pgfplots}
\pgfplotsset{compat=1.15}
\usepackage{mathrsfs}
\usetikzlibrary{arrows}

\usepackage{mathtools}
\usepackage{changepage}
\usepackage{fge}
\usepackage{hyperref}
\usepackage{xcolor}
\usepackage[subnum]{cases}
\usepackage{ulem}
\usepackage{tkz-euclide,amsmath}
\usetikzlibrary{quotes,angles}
\usetikzlibrary{intersections}
\pgfdeclarelayer{bg}
\pgfsetlayers{bg,main}
\usetikzlibrary{positioning}
\usetikzlibrary{calc}

\def\phi{\varphi}

      \def\cO{{\mathcal O}}

\def\ee{{\mathrm e}}

\DeclareMathOperator\dom{{\text{\rm dom\,}}} 

\DeclarePairedDelimiter{\abs}{\lvert}{\rvert}

\DeclareMathOperator{\om}{\Omega}

\DeclareMathOperator{\R}{\mathbb{R}}

\DeclareMathOperator{\diver}{div}
\def\dd{\mathrm{d}}

\renewcommand\setminus{\mathbin{\mathpalette\rsetminusaux\relax}}
\newcommand\rsetminusaux[2]{\mspace{-4mu}
  \raisebox{\rsmraise{#1}\depth}{\rotatebox[origin=c]{-20}{$#1\smallsetminus$}}
 \mspace{-4mu}
}
\newcommand\rsmraise[1]{%
  \ifx#1\displaystyle .8\else
    \ifx#1\textstyle .8\else
      \ifx#1\scriptstyle .6\else
        .45%
      \fi
    \fi
  \fi}

\makeatletter
\newcommand*{\rom}[1]{\expandafter\@slowromancap\romannumeral #1@}
\makeatother

\newtheorem{theorem}{Theorem}[section]
\newtheorem*{thm*}{Theorem}
\newtheorem{proposition}[theorem]{Proposition}
\newtheorem{corollary}[theorem]{Corollary}
\newtheorem{lemma}[theorem]{Lemma}

\theoremstyle{definition}

\theoremstyle{definition}

\newtheorem{example}[theorem]{Example}

\numberwithin{equation}{section}

\title[Inequalities between the lowest eigenvalues]{Inequalities between the lowest eigenvalues of Laplacians with mixed boundary conditions}

\author[N.~Aldeghi]{Nausica Aldeghi}
\author[J.~Rohleder]{Jonathan Rohleder}
\address{Matematiska institutionen \\ Stockholms universitet \\
106 91 Stockholm \\
Sweden}
\email{nausica.aldeghi@math.su.se, jonathan.rohleder@math.su.se}

\begin{document}

\begin{abstract}
The eigenvalue problem for the Laplacian on bounded, planar, convex domains with mixed boundary conditions is considered, where a Dirichlet boundary condition is imposed on a part of the boundary and a Neumann boundary condition on its complement. Given two different such choices of boundary conditions for the same domain, we prove inequalities between their lowest eigenvalues. As a special case, we prove parts of a conjecture on the order of mixed eigenvalues of triangles.
\end{abstract}


\maketitle

\section{Introduction}

On a bounded Lipschitz domain $\Omega \subset \mathbb{R}^2$ we consider the eigenvalue problem for the negative Laplacian $-\Delta_{\Gamma}$ subject to a Dirichlet boundary condition on a non-empty, relatively open part $\Gamma$ of the boundary $\partial \Omega$ and a Neumann boundary condition on its complement $\Gamma^c$, i.e.,
\begin{align*}
 \begin{cases}
 \hspace*{1.6mm} - \Delta u = \lambda u & \text{in}~\Omega, \\
 \hspace{7.4mm} u = 0 & \text{on}~\Gamma, \\
 \nu \cdot \nabla u = 0 & \text{on}~\Gamma^c,
 \end{cases}
\end{align*}
where $\nu$ denotes the unit normal vector field on $\partial \Omega$ pointing outwards. In general, the boundary conditions have to be understood in a weak sense; the operator $-\Delta_\Gamma$ can be defined rigorously via its quadratic form, see Section \ref{sec:preliminaries} below for details. It is self-adjoint in $L^2 (\Omega)$, positive and has a purely discrete spectrum. 

In the present article our aim is to compare the lowest eigenvalue $\lambda_1^\Gamma$ of $- \Delta_\Gamma$ with the corresponding eigenvalue $\lambda_1^{\Gamma'}$ for a different choice $\Gamma'$ of the Dirichlet part of the boundary, for the same domain $\Omega$. In case of an inclusion, $\Gamma' \subset \Gamma$, a variational argument yields
\begin{equation}\label{eq:prelim}
 \lambda_1^{\Gamma'} \leq \lambda_1^\Gamma,
\end{equation}
and the inequality is strict if $\Gamma \setminus \Gamma'$ has a nontrivial interior, see, e.g., \cite[Proposition 2.3]{LR17}; the same is true even for higher eigenvalues. That is, enlarging the part of the boundary where the Dirichlet boundary condition is imposed leads to an increase of the corresponding eigenvalues. 

However, choosing $\Gamma'$ with a smaller length than $\Gamma$ but such that $\Gamma' \setminus \Gamma$ is non-empty does not always guarantee that $\lambda_1^{\Gamma'}$ is smaller than $\lambda_1^\Gamma$. In other words, $\lambda_1^\Gamma$ does not depend monotonously on the length of $\Gamma$; a counterexample is sketched in \cite[Remark 3.3]{S16}. Instead, the validity of \eqref{eq:prelim} depends strongly on the geometry of $\Gamma$ and $\Gamma'$ and how they are located with respect to each other. However, even for very elementary classes of domains this dependence has not been fully understood yet. 

For instance, let $\Omega$ be a triangle and let $\Gamma$ and $\Gamma'$ be two of its sides such that $\Gamma'$ is strictly shorter than $\Gamma$. It has been conjectured by Siudeja in \cite{S16} that $\lambda_1^{\Gamma'} < \lambda_1^\Gamma$ should always hold, but there it was proven only for right triangles satisfying a certain angle restriction. 

In the current paper we are not restricted to triangles but allow general bounded, convex domains with piecewise smooth boundary. We consider two situations:
\begin{itemize}
 \item[(a)] $\Gamma$ and $\Gamma'$ are complementary to each other, i.e.\ the closure of $\Gamma \cup \Gamma'$ is the whole boundary; or
 \item[(b)] $\Gamma$ and $\Gamma'$ do not constitute the whole boundary but the complement of $\Gamma \cup \Gamma'$ has a non-trivial interior.
\end{itemize}
In both cases we assume that $\Gamma'$ is a straight line segment, but we allow $\Gamma$ as well as the possible remainder of the boundary (in case (b)) to be curved. However, we impose conditions on the geometry of $\partial \Omega$ which, in particular, imply that $\Gamma'$ is shorter than $\Gamma$. Under the assumption (a), in Theorem \ref{thm:mainthm1} below we prove the inequality 
\begin{equation}
\label{eq:goal}
\lambda_1^{\Gamma'} < \lambda_1^\Gamma
\end{equation}
if the interior angles of $\partial \Omega$ at the end points of $\Gamma$ are less than $\pi/2$. Under the assumption (b) we are able to prove the same eigenvalue inequality \eqref{eq:goal} if an additional geometric condition on the complement of $\Gamma \cup \Gamma'$ in the boundary is satisfied; see Theorem \ref{thm:mainthm2} below for the details.

An application of our results to triangles resolves a part of Conjecture 1.2 in \cite{S16}. In fact, assume that $\Omega$ is a triangle with shortest side $S$, medium side $M$ and largest side $L$. As consequences of Theorem \ref{thm:mainthm2} and Theorem \ref{thm:mainthm1}, respectively, we get
\begin{equation*}
 \max \{\lambda_1^S, \lambda_1^M \} < \lambda_1^L
\end{equation*}
whenever $\Omega$ is obtuse or right-angled, and
\begin{align*}
 \lambda_1^L < \lambda_1^{S \cup M}
\end{align*}
for any triangle.

The proof of Theorems \ref{thm:mainthm1} and \ref{thm:mainthm2} is variational, using an appropriate directional derivative of an eigenfunction for $\lambda_1^\Gamma$ as test function. This choice of a test function was used before in connection with eigenvalue comparisons for Laplace and Schr\"odinger operators in \cite{LW86,LR17,R20,R21}. In order to estimate the Rayleigh quotient of this test function, we prove an integral identity for the second partial derivatives of Sobolev functions in Lemma \ref{lem:Grisvard} below; it extends a formula from Grisvard's classical book \cite{G85} from polygons to more general, curved domains. 

Finally, we would like to mention that eigenvalues of mixed boundary conditions and inequalities for them have attracted a lot of interest inspired by their importance in the study of, e.g.\ nodal domains of Neumann Laplacian eigenfunctions or the hot spots conjecture, see, e.g., \cite{GWW92,S15}. The limiting behavior of mixed Laplacian eigenvalues if the Dirichlet or Neumann portion shrinks was recently investigated in \cite{FNO21,FNO22}. For further studies of the Laplacian and more general elliptic differential operators with mixed boundary conditions we refer the reader to, e.g., \cite{A11,B94,S02}.

\section{Preliminaries}
\label{sec:preliminaries}

Let us fix some notation and recall some known facts. Throughout the whole paper, $\Omega \subset \R^2$ is a bounded Lipschitz domain, and its boundary~$\partial \Omega$ is piecewise smooth, i.e., $\partial \Omega$ consists of finitely many $C^\infty$-smooth arcs. Note that $\Omega$ being a Lipschitz domain entails that $\partial \Omega$ does not contain any cusp. Moreover, for almost all $x \in \partial \Omega$ there exists a well-defined outer unit normal vector $\nu(x)$ and a unit tangent vector $\tau(x)$ in the direction of positive orientation of the boundary, the points where they are not defined being the corners, i.e.\ the intersections between two consecutive smooth arcs. 
The vector fields $\tau$ and $\nu$ on $\partial \Omega$ are both piecewise smooth with only a finite number of jump discontinuities corresponding to the corners. 

Let us denote by $H^s (\Omega)$, $s > 0$, and $H^s (\partial \Omega)$, $s \in [- 1, 1]$, the Sobolev spaces of order $s$ on $\Omega$ and its boundary, respectively. Here and in the following, $\partial \Omega$ is equipped with the standard surface measure, which we denote by $\sigma$; cf.~\cite{M00}. Recall that there exists a unique bounded trace map from $H^1 (\Omega)$ onto $H^{1/2} (\partial \Omega)$ which continuously extends the mapping 
\begin{align*}
 C^\infty (\overline{\Omega}) \ni u \mapsto u |_{\partial \Omega};
\end{align*}
we usually write $u |_{\partial \Omega}$ for the trace of a function $u \in H^1 (\Omega)$. Moreover, for $u \in H^1 (\Omega)$ satisfying $\Delta u\in L^2(\Omega)$ in the distributional sense we define the normal derivative $\partial_\nu u |_{\partial \Omega}$ of $u$ at $\partial \Omega$ to be the unique element in $H^{- 1/2} (\partial \Omega)$ which satisfies the first Green identity
\begin{align}
\label{eq:Green1}
 \int_\Omega \nabla u \cdot \nabla \overline{v} \, \dd x + \int_\Omega (\Delta u) \overline v \, \dd x = (\partial_\nu u |_{\partial \Omega}, v |_{\partial \Omega} )_{\partial \Omega}, \quad v \in H^1 (\Omega);
\end{align}
here $( \cdot, \cdot)_{\partial \Omega}$ denotes the sesquilinear duality between $H^{1/2} (\partial \Omega)$ and its dual space $H^{- 1/2} (\partial \Omega)$. For sufficiently regular $u$, e.g., $u \in H^2 (\Omega)$, the weakly defined normal derivative $\partial_\nu u |_{\partial \Omega}$ coincides with $\nu \cdot \nabla u |_{\partial \Omega}$ almost everywhere on $\partial \Omega$; in this case the duality in \eqref{eq:Green1} may be replaced by the boundary integral of $\nu \cdot \nabla u |_{\partial \Omega} \overline v |_{\partial \Omega}$.

Let now $\Gamma$ be a relatively open, non-empty subset of $\partial \Omega$. We define $-\Delta_\Gamma$ to be the self-adjoint operator in $L^2(\Omega)$ which corresponds to the non-negative, closed quadratic form 
\begin{equation*}
H^1_{0,\Gamma}(\Omega) \coloneqq \left\{u \in H^1(\Omega) : u|_{\Gamma}=0 \right\} \ni u \mapsto \int_{\Omega} |\nabla u|^2\,\dd x;
\end{equation*}
here $u |_\Gamma = 0$ means that the trace $u |_{\partial \Omega}$ vanishes almost everywhere in $\Gamma$. The operator $- \Delta_\Gamma$ acts as the negative Laplacian, and the functions in its domain satisfy a Dirichlet boundary condition on $\Gamma$ and a Neumann boundary condition in a weak sense on the complement $\Gamma^c = \partial \Omega \setminus \Gamma$; see, e.g., \cite[Section 2]{LR17} for a rigorous discussion of the operator domain. The operator $-\Delta_{\Gamma}$ has a compact resolvent and, hence, its spectrum consists of a discrete sequence of non-negative eigenvalues with finite multiplicities, which converge to $+ \infty$. Its lowest eigenvalue $\lambda_1^\Gamma$ is positive and non-degenerate and can be expressed by the variational principle
\begin{align}
\label{eq:variational}
\lambda_1^\Gamma =\min_{\substack{u \in H^1_{0, \Gamma}(\Omega),\\ u \neq 0}} \frac{\int_{\Omega} |\nabla u|^2\,\dd x}{\int_{\Omega} |u|^2\,\dd x};
\end{align}
$u$ is an eigenfunction of $- \Delta_\Gamma$ corresponding to the eigenvalue $\lambda_1^\Gamma$ if and only if $u$ is a minimizer of \eqref{eq:variational}.

\section{Main results, corollaries, and examples}
\label{sec:mainresults}

In this section we provide the main results of this article and illustrate their assumptions by a number of examples and corollaries. The proofs of the following two theorems are provided in Section \ref{sec:proof} below. In the following first main result we compare the lowest eigenvalues of two configurations that are ``dual'' to each other, i.e.\ Dirichlet and Neumann boundary conditions are interchanged from one to the other.

\begin{theorem}
\label{thm:mainthm1}
Let $\Omega \subset \R^2$ be a bounded, convex domain with piecewise smooth boundary. Furthermore, let $\Gamma, \Gamma' \subset \partial \Omega$ be disjoint, relatively open, non-empty sets such that $\Gamma'$ is a straight line segment and $\overline{\Gamma} \cup \overline{\Gamma'} = \partial \Omega$. Assume that the interior angles of $\partial \Omega$ at both end points of $\Gamma$ are strictly less than $\pi/2$. Then
\begin{equation*}
 \lambda^{\Gamma'}_1 < \lambda^{\Gamma}_1.
\end{equation*}
\end{theorem}

For the special case of non-curved polygons, the statement of Theorem \ref{thm:mainthm1} refines \cite[Theorem 3.1]{R20}, where a non-strict eigenvalue inequality was shown. An example of a domain to which Theorem \ref{thm:mainthm1} applies is shown in Figure \ref{figure1mainthm}. 

\begin{figure}[h]
\begin{tikzpicture}[scale=0.7]
\pgfsetlinewidth{0.8pt}
\node[white] (A) at (0,4) {};
\node[white] (B) at (0,0) {};
\node[white] (C1) at (2.5,3) {};
\node[white] (C2) at (2,1) {};
\draw (A.center) to node[black][left]{$\Gamma'$} (B.center);
\draw[name path=smooth,black] plot[smooth] coordinates {(A) (C1) (C2) (B)};
\node at (barycentric cs:A=1,B=1,C1=1,C2=1){$\Omega$};
\path[name path=middle] (0,2)--(5,2);
\fill [name intersections={of=middle and smooth}](intersection-1) node[right] {$\Gamma$};
\end{tikzpicture} 
\caption{For this domain, $\lambda_1^{\Gamma'} < \lambda_1^\Gamma$ according to Theorem \ref{thm:mainthm1}.}
\label{figure1mainthm}
\end{figure}
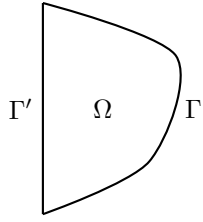

We point out the implications of Theorem \ref{thm:mainthm1} on triangular domains in the following corollary.

\begin{corollary}
Assume that $\Omega$ is a triangle, whose sides are denoted by $S, M$ and $L$ in non-decreasing order according to their lengths. Then the following assertions hold.
\begin{enumerate}
 \item If both angles enclosing $S$ are strictly less than $\pi/2$, then $\lambda_1^S < \lambda_1^{M \cup L}$.
 \item If both angles enclosing $M$ are strictly less than $\pi/2$, then $\lambda_1^M < \lambda_1^{S \cup L}$.
 \item In any case, $\lambda_1^L < \lambda_1^{S \cup M}$.
\end{enumerate}
\end{corollary}

The first two assertions follow immediately from the previous theorem. For the third one, it suffices to note that the angle between the longest side of a triangle and any other of its sides is always below $\pi/2$. The assertion (iii) confirms a part of the aforementioned Conjecture 1.2 in \cite{S16}.

In the following second main result we compare the lowest eigenvalues on a domain $\Omega$ for two configurations: on the one hand, the Dirichlet boundary condition being imposed on an open subset $\Gamma$ of the boundary, on the other hand a Dirichlet boundary condition on an open set $\Gamma'$ that is disjoint with $\Gamma$, but such that $\overline \Gamma \cup \overline {\Gamma'} \subsetneq \partial \Omega$, i.e.\ the two Dirichlet portions for which we compare the lowest eigenvalue do not exhaust the whole boundary. To this end we have to impose the geometric assumption \eqref{eq:miraculous} below on the remaining part $\partial \Omega \setminus (\overline \Gamma \cup \overline {\Gamma'})$ of the boundary. For the following theorem, recall that $\tau$ denotes the unit tangential vector field (along positive orientation of the boundary) and $\nu$ denotes the outer unit normal field.

\begin{theorem}
\label{thm:mainthm2}
Let $\Omega \subset \R^2$ be a bounded, convex domain with piecewise smooth boundary. Furthermore, let $\Gamma, \Gamma' \subset \partial \Omega$ be disjoint, relatively  open, non-empty sets such that $\Gamma'$ is a straight line segment and $\Gamma$ is connected. Assume that the interior angles of $\partial \Omega$ at both end points of $\Gamma$ are strictly less than $\pi/2$. Let $b$ denote the constant outer unit normal vector of $\Gamma'$, and assume further that 
\begin{align}
\label{eq:miraculous}
\text{the function}~(b \cdot \tau)(b \cdot \nu)~\text{is non-increasing along}~\partial \Omega \setminus \Gamma
\end{align}
according to positive orientation. Then
\begin{equation}
\label{eq:maininequalityAgain}
 \lambda^{\Gamma'}_1 < \lambda^{\Gamma}_1.
\end{equation}
\end{theorem}

Some comments are in order. As $\Omega$ is convex with a piecewise smooth boundary and $\Gamma$ is a connected subset of $\partial \Omega$, $\partial \Omega \setminus \Gamma$ consists, except for the corners, of finitely many relatively open smooth arcs $\Sigma_1, \dots, \Sigma_N$, which we can enumerate following the boundary in positive orientation from one end point of $\Gamma$ to the other. We denote the corner points of $\partial \Omega$, i.e.\ the points where two consecutive smooth pieces of $\partial \Omega$ meet, by $P_0, P_1, \dots, P_N$ so that $P_0$ and $P_N$ are the end points of $\Gamma$, and $P_{j-1}$ and $P_j$ are the end points of $\Sigma_j$, see Figure~\ref{fig:notationHelp}. 
\begin{figure}[h]
\begin{tikzpicture}[scale=0.4]
\pgfsetlinewidth{0.8pt}
\node[circle,fill=black,inner sep=0pt,minimum size=2pt,label=below left:{\small {$P_0$}}] (P0) at (9.38,7.24) {};
\node[circle,fill=black,inner sep=0pt,minimum size=2pt,label=below:{\small {$P_1$}}] (P1) at (6.9,6.8) {};
\node[circle,fill=black,inner sep=0pt,minimum size=2pt,label=below:{\small {$P_2$}}] (P2) at (5,6) {};
\node[circle,fill=black,inner sep=0pt,minimum size=2pt,label=right:{\small {$P_3$}}] (P3) at (3,4) {};
\node[circle,fill=black,inner sep=0pt,minimum size=2pt,label=right:{\small {$P_4$}}] (P4) at (3,0) {};
\node[circle,fill=black,inner sep=0pt,minimum size=2pt,label=right:{\small {$P_5$}}] (P5) at (4.22,-1.64) {};
\node[white] (G) at (5.26,-2.36) {}; 
\node[white] (H) at (6.32,-2.76) {}; 
\node[white] (J) at (7.5,-2.98) {}; 
\node[circle,fill=black,inner sep=0pt,minimum size=2pt,label=above left:{\small {$P_6$}}] (P6) at (8.84,-3.18) {};
\begin{pgfonlayer}{bg} 
\pgfsetlinewidth{0.8pt}
\draw[gray](P0.center) to (P1.center);
\node at (8.14,8) {$\Sigma_1$};
\draw[gray](P1.center) to (P2.center);
\node at (5.7,7.45) {$\Sigma_2$};
\draw[gray](P2.center) to node[black][above left]{$\Sigma_3$} (P3.center);
\draw (P3.center) to node[black][left]{$\Gamma'\!=\!\Sigma_4$} (P4.center);
\draw[gray](P4.center) to node[black][below left]{$\Sigma_5$} (P5.center);
\draw[gray] plot[smooth] coordinates {(P5) (G) (H) (J) (P6)};
\draw(6.32,-2.76) node[black][below]{$\Sigma_6$};
\draw (P6.center) to node[black][right]{$\Gamma$} (P0.center);
\end{pgfonlayer}
\node at (barycentric cs:P0=1,P3=1,P4=1,P6=1) {$\Omega$};
\end{tikzpicture}
\caption{A convex domain with piecewise smooth boundary.}
\label{fig:notationHelp}
\end{figure}
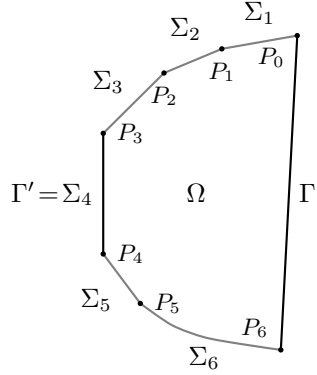
The function $(b \cdot \tau)(b \cdot \nu)$ on $\partial \Omega \setminus \Gamma$ is, like the vector fields $\tau$ and $\nu$, piecewise smooth with only a finite number of jump discontinuities corresponding to the corners. At a corner $P_j$, where two smooth arcs $\Sigma_j$ and $\Sigma_{j+1}$ meet, the monotonicity condition \eqref{eq:miraculous} must be read
\begin{align*}
 \lim_{\Sigma_{j} \ni x \to P_j} (b \cdot \tau (x))(b \cdot \nu (x)) \geq \lim_{\Sigma_{j+1} \ni x \to P_j} (b \cdot \tau (x))(b \cdot \nu (x)).
\end{align*}
By construction, on the straight arc $\Gamma'$,
\begin{equation*}
 (b \cdot \tau)(b \cdot \nu)|_{\Gamma'}=0
\end{equation*}
identically as $b$ is normal to $\Gamma'$. Therefore, the condition \eqref{eq:miraculous} implies $(b \cdot \tau)(b \cdot \nu) \ge 0$ between the end point $P_0$ of $\Gamma$ and the starting point of $\Gamma'$, according to positive orientation of the boundary, and $(b \cdot \tau)(b \cdot \nu) \le 0$ from the end point of $\Gamma'$ to the starting point $P_N$ of $\Gamma$. 

To illustrate the condition \eqref{eq:miraculous}, we provide several examples. First we consider two particularly simple geometries. Within the class of triangles, it turns out that the assumptions of Theorem \ref{thm:mainthm2} are satisfied if and only if the triangle is right or obtuse. The following corollary resolves a part of Conjecture 1.2 in \cite{S16}. For the special case in which $\Omega$ is a right triangle and the smallest angle $\alpha$ satisfies $\pi/6 \leq \alpha \leq \pi/4$, the following was proven in \cite[Theorem 1.1]{S16}. Moreover, a non-strict version of inequality \eqref{eq:maininequalityAgain} was shown for right triangles in \cite[Theorem 4.1]{R20}.

\begin{corollary}
\label{cortriangles}
Let $\Omega$ be an obtuse or right triangle, whose sides are denoted by $S, M$ and $L$ in non-decreasing order according to their lengths. Then
\begin{equation*}
\max{ \{ \lambda_1^{S}, \lambda_1^{M} \} }\ < \lambda_1^{L}
\end{equation*}
holds.
\end{corollary}

\begin{proof}
To show the inequality $\lambda_1^S < \lambda_1^L$, our aim is to apply Theorem \ref{thm:mainthm2} to $L=\Gamma$ and $S=\Gamma'$, see Figure \ref{fig:ObtuseTriangle}.
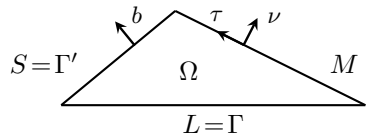
\begin{figure}[h]
\begin{tikzpicture}
\pgfsetlinewidth{0.8pt}
\node[white] (A) at (-6,2) {};
\node[white] (B) at (-4.5,3.25) {};
\node[white] (C) at (-2,2) {};
\draw[name path=sideL] (A.center) to node[black][below]{$L\!=\!\Gamma$} (C.center);
\draw[name path=sideM] (C.center) to (B.center);
\draw[name path=sideS] (B.center) to (A.center);
\node at (-4.32,2.45) {$\Omega$};
\path[name path=newlabels] (-8,2.3)--(0,2.3);
\fill [name intersections={of=newlabels and sideS}](intersection-1) node[above left] {$S\!=\!\Gamma'$};
\fill [name intersections={of=newlabels and sideM}](intersection-1) node[above right] {$M$};
\path[name path=vectors] (-8,2.8)--(0,2.8);
\fill [name intersections={of=vectors and sideS,by=b}](intersection-1);
\node[circle,fill=black,inner sep=0pt,minimum size=0.8pt,sloped] (b') at (b) {};
\node[xshift=0.285cm, yshift=0.09cm] at ($(b'.center)!0.4cm!(b'.129.71513716)$) {\small $b$};
\draw[line width=0.9pt,-stealth] (b'.center) -- ($(b'.center)!0.4cm!(b'.129.71513716)$);
\fill [name intersections={of=vectors and sideM,by=tn}](intersection-1);
\node[circle,fill=black,inner sep=0pt,minimum size=0.8pt,sloped] (tn') at (tn) {};
\draw[line width=0.9pt,-stealth] (tn'.center) -- ($(tn'.center)!0.4cm!(tn'.63.23449906)$) node[right]{\small $\nu$};
\draw[line width=0.9pt,-stealth] (tn'.center) -- ($(tn'.center)!0.4cm!(tn'.153.23449906)$) node[above]{\small $\tau$};
\end{tikzpicture}
\caption{An obtuse triangle; $\max\{\lambda_1^S, \lambda_1^M\} < \lambda_1^L$ holds according to Corollary~\ref{cortriangles}.}
\label{fig:ObtuseTriangle}
\end{figure}
Note that in an obtuse or right triangle, $L$ is the side opposite to the obtuse or right angle, and the interior angles at its end points are smaller than $\pi/2$. Furthermore, denoting by $b$ the unit normal vector of $S$ pointing outwards, due to the obtuse or right angle between $S$ and $M$, both $b \cdot \tau$ and $b \cdot \nu$ are non-negative on $M$, that is, $(b \cdot \tau)(b \cdot \nu) \geq 0$ on $M$. As $(b \cdot \tau)(b \cdot \nu) = 0$ on $S = \Gamma'$, the monotonicity assumption \ref{eq:miraculous} of the theorem is satisfied and we obtain $\lambda_1^S < \lambda_1^L$. The inequality $\lambda^{M}_1 < \lambda^{L}_1$ follows analogously by choosing $M=\Gamma'$.
\end{proof}

Next, we apply Theorem~\ref{thm:mainthm2} to an acute trapezium. Recall that a trapezium is a quadrilateral with at least one pair of parallel sides, called bases, and that it is acute if the two angles adjacent to its longer base are acute; see Figure~\ref{fig:Trapezium}.
\begin{figure}[h]
\begin{tikzpicture}
\pgfsetlinewidth{0.8pt}
\node[white] (LL) at (-2,4) {};
\node[white] (LR) at (-6,4) {};
\node[white] (SL) at (-3.3,5.45) {};
\node[white] (SR) at (-5.18,5.45) {};
\draw (LL.center) to node[black][below]{$L$} (LR.center);
\draw (SL.center) to node[black][above]{$S$} (SR.center);
\draw[gray] (SL.center) to (LL.center);
\draw[gray] (SR.center) to (LR.center);
\node at (barycentric cs:LL=1,LR=1,SL=1,SR=1) {$\Omega$};
\end{tikzpicture}
\caption{An acute trapezium; $\lambda_1^S < \lambda_1^L$ holds according to Corollary \ref{cortrapeze}.}
\label{fig:Trapezium}
\end{figure}
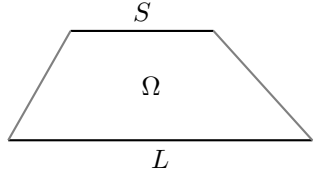

\begin{corollary}
\label{cortrapeze}
If $\Omega$ is an acute trapezium with shorter base $S$ and longer base $L$, then
\begin{equation*}
\lambda_1^S < \lambda_1^L
\end{equation*}
holds.
\end{corollary}

The proof of Corollary~\ref{cortrapeze} is entirely analogous to the proof of Corollary~\ref{cortriangles} and is therefore left to the reader.

The following example may illuminate the implications of condition \eqref{eq:miraculous} on the geometry of $\partial \Omega$ in a more general situation.

\begin{example}
\label{ex:slopeExample}
Let $\Omega$ be a domain of the form as given in Figure~\ref{fig:GenericExample}, where $\Gamma$ and $\Gamma'$ are non-neighboring line segments, the part of $\partial \Omega$ between the end point of $\Gamma$ and the starting point of $\Gamma'$ is polygonal consisting of two line segments, and the part between the end point of $\Gamma'$ and the starting point of $\Gamma$ is a smooth arc.
\begin{figure}[h]
\begin{tikzpicture}[scale=0.45]
\pgfsetlinewidth{0.8pt}
\node[white] (P0) at (8.19766,5.29756) {};
\node[white] (P1) at (6.42332,4.08128) {};
\node[white] (P2) at (5,2) {};
\node[white] (P3) at (5,0) {};
\node[white] (P4) at (7.17312,-3.83347) {};
\node[white] (PC) at (5.54558,-2.1919) {};
\draw[gray] (P0.center) to node[black][above left]{$\Sigma_1$} (P1.center);
\draw[gray] (P1.center) to node[black][above left]{$\Sigma_2$} (P2.center);
\draw (P2.center) to node[left]{$\Gamma'\!=\!\Sigma_3$} (P3.center);
\draw (P4.center) to node[right]{$\Gamma$} (P0.center);
\draw[gray] plot[smooth] coordinates {(P3) (PC) (P4)};
\draw(5.54558,-2.1919) node[black][left]{$\Sigma_4$};
\node at (barycentric cs:P0=1,P2=1,P3=1,P4=1) {$\Omega$};
\end{tikzpicture}
\caption{The setting of Example \ref{ex:slopeExample}.}
\label{fig:GenericExample}
\end{figure}
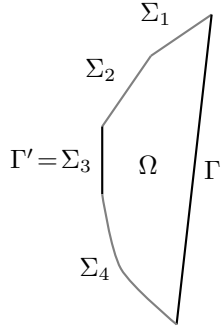
Following the notation introduced above, we call the two line segments, ordered according to positive orientation, $\Sigma_1$ and $\Sigma_2$, and the smooth arc $\Sigma_4$. We assume without loss of generality that $\Omega$ is rotated in such a way that $\Gamma'$ is parallel to the second coordinate axis; then the normal vector of $\Gamma'$ pointing outwards is $b = (-1, 0)^\top$ and the monotonicity condition \eqref{eq:miraculous} reads
\begin{align*}
 \tau_1 \tau_2~\text{is non-increasing along}~\partial \Omega \setminus \Gamma,
\end{align*}
according to positive orientation. In this position, the smooth arc $\Sigma_4$ is the graph of a smooth, convex, non-increasing function $\phi : [0, 1] \to \R$, and the unit tangent vector at a point $x = (t, \phi (t))^\top \in \Sigma_4$ equals
\begin{equation*}
\tau (x) = \frac{1}{\sqrt{1 + (\phi'(t))^2}} \begin{pmatrix} 1  \\ \phi'(t) \end{pmatrix},
\end{equation*}
so that
\begin{equation*}
\tau_1 (x) \tau_2 (x) = \frac{\phi'(t)}{1+(\phi'(t))^2}, \quad x = (t, \phi (t))^\top \in \Sigma_4.
\end{equation*}
This is non-increasing in $t$ if and only if its derivative is non-positive, meaning
\begin{align*}
 \frac{\phi''(t) (1 - {\phi'(t)}^2)}{(1 + {\phi'(t)}^2{)}^2} \leq 0, \quad t \in [0, 1],
\end{align*}
equivalently, $\phi' (t) \leq - 1$ or $\phi'' (t) = 0$ for all $t \in [0, 1]$, where we have employed convexity of $\phi$. Hence, for the smooth arc $\Sigma_4$, condition \eqref{eq:miraculous} translates into
\begin{align*}
\phi' (t) \leq - 1, \quad t \in [0, 1],
\end{align*}
as long as $\Sigma_4$ is not a straight line segment.

For the two straight upper sides $\Sigma_1$ and $\Sigma_2$, by an analogous reasoning, condition \eqref{eq:miraculous} can be expressed in terms of their slopes $m_1 > 0$ and $m_2 > 0$, respectively, as
\begin{align*}
 \frac{m_1}{1 + m_1^2} \geq \frac{m_2}{1 + m_2^2}.
\end{align*}
As the convexity of $\Omega$ only allows $m_1 \leq m_2$, this is true if and only if 
\begin{align*}
 \frac{1}{m_2} \leq m_1~\text{and}~m_2 \geq 1.
\end{align*}
\end{example}

The next example shows the limitations of assumption \eqref{eq:miraculous}.

\begin{example}
Let $\Omega$ be a quadrilateral as given in Figure~\ref{fig:almostTriangle},
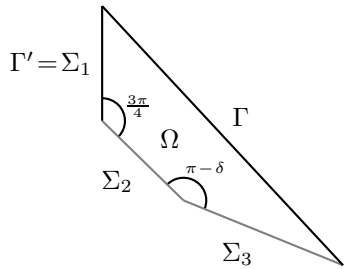
\begin{figure}[h]
\begin{tikzpicture}[scale=0.6]
\pgfsetlinewidth{0.8pt}
\node[white] (P0) at (1,2.53624) {};
\node[white] (P1) at (1,0) {};
\node[white] (P2) at (2.78221,-1.75793) {};
\node[white] (P3) at (6.28,-3.19) {};
\tkzMarkAngle[size=0.5, mark=none](P2,P1,P0)
\tkzLabelAngle[pos=0.8](P2,P1,P0){\tiny $\frac{3 \pi}{4}$}
\tkzMarkAngle[size=0.5, mark=none](P3,P2,P1)
\tkzLabelAngle[pos=0.8](P3,P2,P1){\tiny $\pi \! - \! \delta$}
\draw (P0.center) to node[black][left]{$\Gamma'\!=\!\Sigma_1$} (P1.center);
\draw[gray] (P1.center) to node[black][below left]{$\Sigma_2$}(P2.center);
\draw[gray] (P2.center) to node[black][below left]{$\Sigma_3$}(P3.center);
\draw (P3.center) to node[black][above right]{$\Gamma$} (P0.center);
\node (bary) at (barycentric cs:P0=1,P1=1,P2=1,P3=0.7) {$\Omega$};
\end{tikzpicture}
\caption{Inequality~\eqref{eq:maininequalityAgain} does not hold for this quadrilateral.}
\label{fig:almostTriangle}
\end{figure}
with the slopes of $\Sigma_2$ and $\Sigma_3$, respectively, chosen as
\begin{align*}
 m_2 = - 1 \qquad \text{and} \qquad m_3 = - \tan \left( \frac{\pi}{4} - \delta \right)
\end{align*}
for some sufficiently small $\delta > 0$. In the same way as in the previous example, condition \eqref{eq:miraculous} can be rewritten
\begin{align*}
 - \frac{1}{2} = \frac{m_2}{1 + m_2^2} \geq \frac{m_3}{1 + m_3^2} = - \frac{\tan (\frac{\pi}{4} - \delta)}{1 + \tan^2 (\frac{\pi}{4} - \delta)}.
\end{align*}
However, since the function $x \mapsto \frac{x}{1 + x^2}$ takes its strict global maximum at $x = 1$, with value $\frac{1}{2}$, this is wrong for all small $\delta > 0$. Thus Theorem \ref{thm:mainthm2} cannot be applied. At the same time, sending $\delta$ to zero, $\Omega$ converges to an obtuse triangle, and by Corollary \ref{cortriangles} we have $\lambda_1^{\Gamma'} < \lambda_1^\Gamma$ in the case $\delta = 0$. By continuity, the same inequality holds for all sufficiently small $\delta > 0$. This indicates that the assumption \eqref{eq:miraculous} of Theorem~\ref{thm:mainthm2} is not optimal.
\end{example}

We finally point out that the assumption in Theorem \ref{thm:mainthm2} on the interior angles of $\partial \Omega$ at the end points of $\Gamma$ to be strictly less than $\pi / 2$ cannot be omitted to obtain the strict inequality $\lambda_1^{\Gamma'} < \lambda_1^\Gamma$. For instance, take $\Omega$ to be any rectangle and let $\Gamma$ and $\Gamma'$ be opposite sides. Then, by symmetry, $\lambda_1^{\Gamma} = \lambda_1^{\Gamma'}$, that is, a strict inequality does no longer hold. On the other hand, it may be conjectured that Theorem~\ref{thm:mainthm1} is true even without the assumption on the angles.

\section{Proofs of the main results}
\label{sec:proof}

In this section we prove the main results of this paper. We start by collecting a few useful statements. The first one is a sufficient criterion for functions in the domain $\dom (- \Delta_\Gamma)$ of $- \Delta_\Gamma$ to belong to the Sobolev space $H^2 (\Omega)$. Such regularity holds under assumptions on the angles at the corners at which the transition between Dirichlet and Neumann condition takes place. The following regularity result will be instrumental to the proof of the main result; in fact, it will be used to make sure that the test function constructed there belongs to the Sobolev space $H^1 (\Omega)$.

\begin{proposition}
\label{prop:regularity}
Assume that $\Omega \subset \R^2$ is a bounded Lipschitz domain with piecewise smooth boundary and that $\Gamma, \Gamma^c \subset \partial \Omega$ are relatively open such that $\overline{\Gamma} \cup \overline{\Gamma^c} = \partial \Omega$ holds. Moreover, assume that all angles at which  $\Gamma$ and $\Gamma^c$ meet are strictly less than $\pi/2$ and that the angles at all interior corners of $\Gamma$ and $\Gamma^c$ are less or equal $\pi$. Then 
\begin{align*}
 \dom{(-\Delta_\Gamma)} \subset  H^2(\om).
\end{align*}
\end{proposition}

If $\Omega$ is a polygon, the statement follows from \cite[Theorem 4.4.3.3 and Lemma 4.4.1.4]{G85}. Else, if $\partial \Omega$ has non-trivial curvature at some point, $H^2$ regularity of the functions in $\dom (- \Delta_\Gamma)$ away from the corners where the transition between Dirichlet and Neumann condition takes place is a standard result, see e.g.\ the survey \cite{KO83}. As for these corners, they can be locally mapped conformally into a sector of the same angle, see \cite{A81} or \cite[Section 5]{W70}.

Secondly, we make the following observation, which is a simple consequence of a unique continuation principle. It will be used to show that the eigenvalue inequalities in our main results are always strict; for the convenience of the reader we provide a short proof.

\begin{lemma}
\label{lem:continuation principle}
Let $\Omega$ be a bounded, connected Lipschitz domain, let $\lambda \in \mathbb{R}$ and let $u \in H^1(\Omega)$ be such that $-\Delta u = \lambda u$ holds in the distributional sense. If $\Lambda \subset \partial \Omega$ is a relatively open, non-empty subset such that $u|_{\Lambda}=0$ and $\partial_{\nu} u|_{\Lambda}=0$ then $u=0$ identically on $\Omega$.
\end{lemma}

\begin{proof}
As $\Lambda$ is relatively open in $\partial \Omega$, we may choose another bounded Lipschitz domain $\widetilde \Omega$ such that $\Omega \subset \widetilde \Omega$, $\widetilde \Omega \setminus \Omega$ contains an open ball $\cO$, and $\partial \Omega \setminus \partial \widetilde \Omega \subset \Lambda$. Then the function $\widetilde u$ obtained from $u$ by extending it by zero to $\widetilde \Omega \setminus \Omega$ belongs to $H^1 (\widetilde \Omega)$ and $- \Delta \widetilde u = \lambda \widetilde u$ holds on $\widetilde \Omega$ in the distributional sense due to $u |_\Lambda = \partial_\nu u |_\Lambda = 0$; cf.\ \cite[Lemma 3.1]{R14}. As $\widetilde u$ vanishes identically on $\cO$, a classical unique continuation statement yields $\widetilde u = 0$ identically in $\widetilde \Omega$, see, e.g., \cite{W93}. Hence, $u = 0$ in $\Omega$ as $\Omega \subset \widetilde \Omega$.
\end{proof}

Finally, the following integration-by-parts formula is crucial for the proof of the main results. In order to formulate it, we introduce the signed curvature of $\partial \Omega$ with respect to the outer unit normal $\nu$, defined at each point of $\partial \Omega$ except at the corners, see for instance \cite[Exercise 8, Section 2.3]{O06} or \cite[Section 2.2]{PR10}. It can be expressed as 
\begin{equation*}
\kappa = \tau' \cdot \nu,
\end{equation*}
where $\tau$ is the unit tangent vector field along the boundary and the derivative $\tau'$ is to be understood piecewise via an arclength parametrization of the boundary in positive direction. Later on we will consider convex domains, for which $\kappa(x) \le 0$ holds for almost all $x \in \partial \Omega$.

\begin{lemma}
\label{lem:Grisvard}
Assume that $\Omega \subset \mathbb{R}^2$ is a bounded Lipschitz domain with piecewise smooth boundary consisting, except for the corners, of finitely many smooth arcs $\Gamma_1, \dots, \Gamma_N$, and let $\kappa$ denote the curvature of $\partial \Omega$ w.r.t.\ the unit normal pointing outwards. Then 
\begin{align*}
\int_{\Omega} (\partial_1^2 u)(\partial_2^2 u)\,\textup{d}x = \int_{\Omega} (\partial_1 \partial_2 u)^2\,\textup{d}x - \frac{1}{2} \int_{\partial \Omega} \kappa \abs{\nabla u}^2 \, \textup{d} \sigma 
\end{align*}
holds for all real-valued $u \in V^2 (\Omega)$, where
\begin{align*}
 V^2 (\Omega) = \Big\{ w \in H^2 (\Omega) : \text{on each}~\Gamma_j,~w |_{\Gamma_j} = 0~\text{or}~\partial_\nu w |_{\Gamma_j} = 0 \Big\},
\end{align*}
i.e.\ for all functions in $H^2 (\Omega)$ which satisfy a Dirichlet or Neumann boundary condition on each smooth arc.
\end{lemma}

\begin{proof}
Let first $\mathbf{u} = (u_1, u_2)^\top \in C^\infty (\overline \Omega, \, \mathbb{R}^2)$ be a vector field whose components have compact supports in $\overline{\Omega}$ which do not contain any corner of $\partial \Omega$; up to an approximation, $\mathbf{u}$ will later play the role of $\nabla u$. Two consecutive applications of integration by parts yield
\begin{align}
\label{eq:GrisvardLemmaBP}
\begin{split}
 \int_\Omega (\partial_{1} u_1) (\partial_{2} u_2 ) \, \textup{d} x 
 & = - \int_\Omega u_1 \partial_1 \partial_{2} u_2 \, \textup{d} x + \int_{\partial \Omega} u_1 \nu_1 \partial_{2} u_2 \, \textup{d} \sigma \\ 
  & = \int_\Omega (\partial_{2} u_1) (\partial_{1} u_2) \, \textup{d} x + \int_{\partial \Omega} ( u_1 \nu_1 \partial_{2} u_2 - u_1 \nu_2 \partial_{1} u_2 ) \, \textup{d} \sigma \\
 & = \int_\Omega (\partial_{2} u_1) (\partial_{1} u_2) \, \textup{d} x + \int_{\partial \Omega} u_1 \partial_\tau u_2 \, \textup{d} \sigma, 
\end{split}
\end{align}
where $\partial_\tau$ denotes the derivative in the direction of the tangential vector field $\tau$ on $\partial \Omega$. In the following, we use the identities $\nu_1 \tau_2 - \nu_2 \tau_1 = 1$ and
\begin{align*}
 u_j = \ee_j \cdot \mathbf{u} = (\tau \cdot \ee_j ) \, (\tau \cdot \mathbf{u}) + (\nu \cdot \ee_j) \, (\nu \cdot \mathbf{u}) = \tau_j \, (\tau \cdot \mathbf{u}) + \nu_j \, (\nu \cdot \mathbf{u}),
\end{align*}
which holds for each standard basis vector $\ee_j$ everywhere on $\partial \Omega$ except at the corners. Inspired by the proof of \cite[Proposition 2.1]{LW86}, we rewrite the integrand in the boundary integral in \eqref{eq:GrisvardLemmaBP},
\begin{align*}
u_1 \partial_\tau u_2
 & = \frac{1}{2} \big( \partial_\tau (u_1 u_2) - (\partial_\tau u_1) u_2 \big) + \frac{1}{2} u_1 \partial_\tau u_2 \\ 
 & = \frac{1}{2} \partial_\tau (u_1 u_2) + \frac{1}{2} \Big( \big(\tau_1 (\tau \cdot \mathbf{u}) + \nu_1 (\nu \cdot \mathbf{u}) \big) \partial_\tau \big(\tau_2 (\tau \cdot \mathbf{u})  + \nu_2 (\nu \cdot u)\big) \\
 & \quad - \big(\tau_2 (\tau \cdot \mathbf{u}) + \nu_2 (\nu \cdot \mathbf{u}) \big) \partial_\tau \big(\tau_1 (\tau \cdot \mathbf{u})  + \nu_1 (\nu \cdot \mathbf{u}) \big) \Big) \\
 & = \frac{1}{2} \partial_\tau (u_1 u_2) + \frac{1}{2} \big((\nu \cdot \mathbf{u}) \partial_\tau (\tau \cdot \mathbf{u}) - (\tau \cdot \mathbf{u}) \partial_\tau (\nu \cdot \mathbf{u}) \big) \\
 & \quad + \frac{1}{2} \Big( \big(\tau_1 (\tau \cdot \mathbf{u}) + \nu_1 (\nu \cdot \mathbf{u})\big) \big(\tau_2' (\tau \cdot \mathbf{u}) + \nu_2' (\nu \cdot \mathbf{u})\big) \\
 & \quad - \big(\tau_2 (\tau \cdot \mathbf{u}) + \nu_2 (\nu \cdot \mathbf{u})\big) \big(\tau_1' (\tau \cdot \mathbf{u}) + \nu_1' (\nu \cdot \mathbf{u}) \big) \Big),
\end{align*}
which holds everywhere on $\partial \Omega$, since the supports of $u_1$ and $u_2$ do not contain any corners. From the relations 
\begin{align*}
\tau_1 \tau_2' - \tau_2 \tau_1' = - \kappa = \nu_1 \nu_2' - \nu_2 \nu_1'
\end{align*}
and 
\begin{align*}
\nu_1 \tau_2' - \nu_2 \tau_1' = 0 = \tau_1 \nu_2' - \tau_2 \nu_1'
\end{align*}
we finally obtain
\begin{align*}
u_1 \partial_\tau u_2 & = \frac{1}{2} \partial_\tau (u_1 u_2) - \frac{\kappa}{2} \left( (\tau \cdot \mathbf{u})^2 + (\nu \cdot \mathbf{u})^2 \right) \\ & \quad + \frac{1}{2} \big((\nu \cdot \mathbf{u}) \partial_\tau (\tau \cdot \mathbf{u}) - (\tau \cdot \mathbf{u}) \partial_\tau (\nu \cdot \mathbf{u}) \big)
\end{align*}
everywhere on $\partial \Omega$. Integrating over the closed boundary curve, the first term on the right-hand side vanishes and we get
\begin{align*}
 \int_{\partial \Omega} u_1 \partial_\tau u_2 \,\textup{d} \sigma 
 & = - \frac{1}{2} \int_{\partial \Omega} \kappa \abs{\mathbf{u}}^2 \, \textup{d} \sigma \\
& \phantom{=} + \frac{1}{2} \sum_{j=1}^N \int_{\Gamma_j} \big((\nu \cdot \mathbf{u}) \partial_\tau (\tau \cdot \mathbf{u}) - (\tau \cdot \mathbf{u}) \partial_\tau (\nu \cdot \mathbf{u}) \big) \, \textup{d}\sigma
\end{align*}
which plugged into \eqref{eq:GrisvardLemmaBP} yields 
\begin{align}
\label{eq:GrisvardLemmaAlmost}
\begin{split}
\int_\Omega (\partial_{1} u_1 ) (\partial_{2} u_2 ) \, \textup{d} x  = & \int_\Omega (\partial_{2} u_1) (\partial_{1} u_2) \, \textup{d} x - \frac{1}{2} \int_{\partial \Omega} \kappa \abs{\mathbf{u}}^2 \, \textup{d} \sigma \\
& \phantom{=} + \frac{1}{2} \sum_{j=1}^N \int_{\Gamma_j} \big((\nu \cdot \mathbf{u}) \partial_\tau (\tau \cdot \mathbf{u}) - (\tau \cdot \mathbf{u}) \partial_\tau (\nu \cdot \mathbf{u}) \big) \, \textup{d}\sigma
\end{split}
\end{align}
for all $\mathbf{u} = (u_1, u_2)^\top \in C^\infty (\overline \Omega, \, \R^2)$ whose supports do not contain any corners of~$\partial \Omega$. 

Now if $u \in V^2(\Omega)$, then $\nabla u \in H^1(\Omega, \, \R^2)$ and, further, $u$ satisfies a Dirichlet or a Neumann boundary condition on each smooth arc of $\partial \Omega$, so that in almost each point of $\partial \Omega$ either the tangential or the normal derivative of $u$ respectively vanish. Hence, $(\partial_\nu u)(\partial_\tau \partial_\tau u) = (\partial_\tau u)(\partial_\tau \partial_\nu u) = 0$ constantly for each $j=1, \ldots, N$. By a capacity argument the vector fields in $C^\infty (\overline{\Omega}, \, \R^2)$ whose supports do not contain the corners of $\partial \Omega$ are dense in $H^1(\Omega, \, \R^2)$, see \cite[Chapter 8, Corollary 6.4]{EE18}. Hence there exists a sequence $(\mathbf{u}^{(k)})_k$ in $C^\infty (\overline{\Omega}, \, \R^2)$ of vector fields whose supports do not contain the corners of $\partial \Omega$ such that $\mathbf{u}^{(k)} \to \nabla u$ in $H^1(\Omega, \, \R^2)$ as $k \to \infty$. In particular,  one has
\begin{align*}
(\nu \cdot \mathbf{u}^{(k)}) \partial_\tau (\tau \cdot \mathbf{u}^{(k)}) \to 0 \quad \text{and} \quad (\tau \cdot \mathbf{u}^{(k)}) \partial_\tau (\nu \cdot \mathbf{u}^{(k)}) \to 0
\end{align*}
in $L^1(\Gamma_j)$, $j = 1, \dots, N$, as well as
\begin{align*}
 \mathbf{u}^{(k)} |_{\partial \Omega} \to \nabla u |_{\partial \Omega}
\end{align*}
in $H^{1/2} (\partial \Omega)$. Therefore, from \eqref{eq:GrisvardLemmaAlmost} we obtain
\begin{align*}
 \int_{\Omega} (\partial_1^2 u)(\partial_2^2 u)\,\textup{d}x & = \lim_{k \to \infty}  \int_\Omega (\partial_{1} \mathbf{u}^{(k)}_1 ) (\partial_{2} \mathbf{u}^{(k)}_2 ) \, \textup{d} x \\
 & = \lim_{k \to \infty} \int_\Omega (\partial_{2} \mathbf{u}^{(k)}_1) (\partial_{1} \mathbf{u}^{(k)}_2) \, \textup{d} x - \frac{1}{2} \lim_{k \to \infty} \int_{\partial \Omega} \kappa \abs{\mathbf{u}^{(k)}}^2 \, \textup{d} \sigma \, \textup{d} \sigma \\
 & \phantom{=} + \frac{1}{2} \sum_{j=1}^N \lim_{k \to \infty}  \int_{\Gamma_j} \big((\nu \cdot \mathbf{u}^{(k)}) \partial_\tau (\tau \cdot \mathbf{u}^{(k)}) - (\tau \cdot \mathbf{u}^{(k)}) \partial_\tau (\nu \cdot \mathbf{u}^{(k)}) \big) \, \textup{d}\sigma \\
 & = \int_{\Omega} (\partial_1 \partial_2 u)^2\,\textup{d}x - \frac{1}{2} \int_{\partial \Omega} \kappa \abs{\nabla u}^2 \, \textup{d} \sigma.
\end{align*}
This completes the proof.
\end{proof}

We proceed to the proof of the main results.

\begin{proof}[Proof of Theorem~\ref{thm:mainthm1} and Theorem \ref{thm:mainthm2}]
We prove the two theorems simultaneously and divide the proof into three steps. 

{\bf Step 1.} This step contains some preparations. Let $\Gamma' \subset \partial \Omega$ be an open straight line segment and $\Gamma \subset \partial \Omega$ be relatively open and connected such that $\Gamma \cap \Gamma' = \emptyset$ and the interior angles of $\partial \Omega$ at both end points of $\Gamma$ are strictly less than $\pi/2$. We denote by $b$ the constant outer unit normal vector of $\Gamma'$. Let $u$ be a real-valued eigenfunction of $-\Delta_{\Gamma}$ corresponding to the eigenvalue $\lambda_1^{\Gamma}$, and let $v = b \cdot \nabla u$, the directional derivative of $u$ in the direction of $b$. As the angles adjacent to $\Gamma$ are less than $\pi/2$ and $\Omega$ is convex, it follows from Proposition \ref{prop:regularity} that $v \in H^1(\Omega)$. Moreover, the Neumann boundary condition imposed on $u$ on $\Gamma'$ implies $v |_{\Gamma'} = 0$, i.e.\ $v \in H_{0, \Gamma'}^1 (\Omega)$. Note that $v$ is non-trivial since $b \cdot \nabla u=0$ identically on $\Omega$ together with $u=0$ on $\Gamma$ would imply $u=0$ on $\Omega$, since $\Gamma$ is not a straight line segment orthogonal to $\Gamma'$, due to the angle requirement in the theorems and convexity of $\Omega$.

Our first aim is to prove that
\begin{align}
\label{eq:almost}
 \frac{\int_\Omega |\nabla v|^2 \, \textup{d} x}{\int_\Omega |v|^2 \,\textup{d} x} & \leq \lambda_1^\Gamma.
\end{align}
First of all, integration by parts yields
\begin{align}
\label{eq:half}
\begin{split}
\lambda_1^\Gamma \int_\Omega |v|^2 \, \textup{d} x & =  \lambda_1^\Gamma \int_\Omega \nabla u \cdot b b^\top \nabla u \, \textup{d} x \\
& = \lambda_1^\Gamma \left( - \int_\Omega u \diver \left(b b^\top \nabla u \right) \, \textup{d} x + \int_{\partial \Omega} u b b^\top \nabla u \cdot \nu \, \textup{d} \sigma \right) \\
& = \int_\Omega \Delta u \diver \left(b b^\top \nabla u \right) \, \textup{d} x + \lambda_1^\Gamma \int_{\partial \Omega} u (b \cdot \nabla u)(b \cdot \nu) \, \textup{d} \sigma.
\end{split}
\end{align}
The domain integral on the right-hand side may be rewritten with the help of Lemma~\ref{lem:Grisvard},
\begin{align}
\label{eq:otherhalf}
\begin{split}
 \int_\Omega \Delta u \diver \left(b b^\top \nabla u \right) \, \textup{d} x & = \int_\Omega (\Delta u ) \left( b_1^2 \partial_1^2 u + 2 b_1 b_2 \partial_1 \partial_2 u + b_2^2 \partial_2^2 u \right) \, \textup{d} x \\
 & = \int_\Omega b_1^2 |\nabla \partial_1 u|^2 \, \textup{d} x + 2 \int_\Omega b_1 b_2 \nabla \partial_1 u \cdot \nabla \partial_2 u \, \textup{d} x \\
 & \quad + \int_\Omega b_2^2 |\nabla \partial_2 u|^2 \, \textup{d} x - \frac{1}{2} \left(b_1^2 + b_2^2 \right) \int_{\partial \Omega} \kappa |\nabla u|^2 \, \textup{d} \sigma \\
 & = \int_\Omega |\nabla v|^2 \, \textup{d} x - \frac{1}{2}  \int_{\partial \Omega} \kappa |\nabla u|^2 \, \textup{d} \sigma.
\end{split}
\end{align}
By plugging \eqref{eq:otherhalf} into \eqref{eq:half} we obtain 
\begin{align}
\label{eq:fundeq}
\lambda_1^\Gamma \int_\Omega |v|^2 \, \textup{d} x = \int_\Omega |\nabla v|^2 \, \textup{d} x +  \lambda_1^\Gamma \int_{\partial \Omega} u (b \cdot \nabla u)(b \cdot \nu) \, \textup{d} \sigma - \frac{1}{2}  \int_{\partial \Omega} \kappa |\nabla u|^2 \, \textup{d} \sigma.
\end{align}

{\bf Step 2.} We will now prove Theorem~\ref{thm:mainthm1}. Indeed, if $\overline{\Gamma} \cup \overline{\Gamma'} = \partial \Omega$ then the first boundary integral vanishes as $u$ vanishes on $\Gamma$ and the normal derivative $\partial_\nu u = \nu \cdot \nabla u$ vanishes on $\Gamma'$, where $\nu=b$ constantly. The convexity of $\Omega$ implies $\kappa \leq 0$ almost everywhere on $\partial \Omega$, therefore we obtain 
\begin{align}
\label{eq:yes!}
\begin{split}
 \lambda_1^{\Gamma'} \leq \frac{\int_\Omega |\nabla v|^2 \, \textup{d} x}{\int_\Omega |v|^2 \, \textup{d} x} \leq \frac{\int_\Omega |\nabla v|^2 \, \textup{d} x - \frac{1}{2}  \int_{\partial \Omega} \kappa |\nabla u|^2 \, \textup{d} \sigma}{\int_\Omega |v|^2 \, \textup{d} x} = \lambda_1^\Gamma.
\end{split}
\end{align}
In particular, \eqref{eq:almost} follows under the assumptions of Theorem~\ref{thm:mainthm1}. From this and the variational principle \eqref{eq:variational} we obtain the inequality $\lambda_1^{\Gamma'} \leq \lambda_1^\Gamma$. 

To conclude the proof of Theorem~\ref{thm:mainthm1}, we distinguish two cases. First, assume that there exists a non-corner point of the boundary where $\partial \Omega$ has non-zero curvature. As $\Omega$ is convex and $\partial \Omega$ is piecewise smooth, this implies the existence of a relatively open subset $\Lambda \subset \partial \Omega$ on which $\kappa$ is uniformly negative, and in fact $\Lambda \subset \Gamma$ as $\Gamma'$ is a straight line segment. If we assume for a contradiction that $\lambda_1^{\Gamma'} = \lambda_1^\Gamma$ then, by \eqref{eq:yes!}, 
\begin{align}\label{eq:curvatureZero}
 \int_{\partial \Omega} \kappa |\nabla u|^2 \, \textup{d} \sigma = 0,
\end{align}
which in particular implies $\nabla u = 0$ almost everywhere on $\Lambda$. Especially, both $u$ and $\partial_\nu u$ vanish on $\Lambda$, and Lemma \ref{lem:continuation principle} implies $u = 0$ constantly in $\Omega$, a contradiction. 

In the second case $\partial \Omega$ is piecewise straight, i.e. $\Omega$ is a polygon. Then $\Gamma$ consists of finitely many relatively open straight line segments. We denote by $\Gamma_1$ the segment which precedes $\Gamma'$ according to positive orientation of the boundary, and by $P$ the end point shared by $\Gamma'$ and $\Gamma_1$, as given in Figure~\ref{fig:strictPolygon}. Recall that the interior angle of $\partial \Omega$ at which $\Gamma'$ and $\Gamma_1$ meet is by assumption strictly less than $\pi/2$, and that $b$ is the outer unit normal vector to any point of $\Gamma'$.
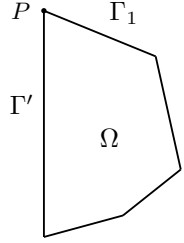
\begin{figure}[h]
\begin{tikzpicture}[scale=0.5]
\pgfsetlinewidth{0.7pt}
\node[circle,fill=black,inner sep=0pt,minimum size=2pt,label=left:{{\small $P$}}] (P) at (4,7) {};
\node[white] (Q) at (4,1) {};
\node[white] (R) at (6.08,1.57) {};
\node[white] (S) at (7.6,2.79) {};
\node[white] (T) at (6.94,5.79) {};
\draw (P.center) to node[black][above left]{$\Gamma'$} (Q.center);
\draw (Q.center) to (R.center);
\draw (R.center) to (S.center);
\draw (S.center) to (T.center);
\draw (T.center) to node[black][above right]{$\Gamma_1$}(P.center);
\node at (barycentric cs:P=1,Q=1,R=1,S=1,T=1) {$\Omega$};
\end{tikzpicture}
\caption{The positions of $\Gamma_1$ and $P$ as used in Step 2 of the proof.}
\label{fig:strictPolygon}
\end{figure}
At almost each point of $\Gamma_1$, we can express the vector $b$ as a linear combination of the unit vectors $\tau$ and $\nu$ which are tangential respectively normal to $\Gamma_1$,
\begin{equation}
\label{eq:bLinearComb}
 b = (b \cdot \tau) \, \tau + (b \cdot \nu) \, \nu,
\end{equation}
so that on $\Gamma_1$ we obtain
\begin{align}
\label{eq:reint2}
v = b \cdot \nabla u = (b \cdot \tau) \, \tau \cdot \nabla u + (b \cdot \nu) \, \nu \cdot \nabla u = (b \cdot \tau) \, \partial_\tau u + (b \cdot \nu) \, \partial_\nu u.
\end{align}
From this and the fact that $0 = - \lambda_1^\Gamma u = \Delta u = \partial_\tau \partial_\tau u + \partial_\nu \partial_\nu u = \partial_\nu \partial_\nu u$ on $\Gamma_1$, we get
\begin{align}
\label{eq:reint4}
\begin{split}
 \partial_\nu v & =  (b \cdot \tau) \, \partial_\nu \partial_\tau u + (b \cdot \nu) \, \partial_\nu \partial_\nu u =  (b \cdot \tau) \, \partial_\nu \partial_\tau u \quad \text{on}~\Gamma_1.
\end{split}
\end{align}
Now, for a contradiction, assume that $\lambda_1^{\Gamma'} = \lambda_1^\Gamma$ holds. By \eqref{eq:yes!} this implies that $v = b \cdot \nabla u$ is an eigenfunction of $- \Delta_{\Gamma'}$ corresponding to $\lambda_1^{\Gamma'}$. Then $\partial_\nu v =0$ on $\Gamma$ and in particular on $\Gamma_1$, so that the left-hand side of \eqref{eq:reint4} vanishes on $\Gamma_1$. As $\Gamma_1$ and $\Gamma'$ cannot be parallel, we have $b \cdot \tau \neq 0$ on $\Gamma_1$, and \eqref{eq:reint4} yields that $0 = \partial_\nu \partial_\tau u = \partial_\tau \partial_\nu u$ on $\Gamma_1$, i.e., there exists a constant $c \in \mathbb{R}$ such that on $\Gamma_1$
\begin{equation}
\label{eq:normalDerConstant}
\partial_\nu u = c.
\end{equation}
Moreover, the corner $P$ is a critical point of $u$ as $u$ satisfies a Neumann boundary condition on $\Gamma'$ and a Dirichlet boundary condition on $\Gamma_1$ and the two sides are not perpendicular. Thus $\nabla u(P) =  0$, which combined with \eqref{eq:normalDerConstant} gives $\partial_\nu u = 0$ on $\Gamma_1$. Since $u=0$ on $\Gamma_1$, Lemma~\ref{lem:continuation principle} yields $u = 0$ constantly in $\Omega$, a contradiction. We have thus proven Theorem~\ref{thm:mainthm1}.

{\bf Step 3.} Now we turn to the proof of Theorem \ref{thm:mainthm2} based on \eqref{eq:fundeq}. In fact, as $\kappa (x) \leq 0$ for almost all $x \in \partial \Omega$, inequality \eqref{eq:almost} follows if we can show that
\begin{align}
\label{eq:done}
\int_{\partial \Omega} u (b \cdot \nabla u)(b \cdot \nu) \, \textup{d} \sigma \geq 0.
\end{align}
Note first that the integrand vanishes on $\Gamma$ as $u |_\Gamma = 0$ constantly. To estimate the integral over the remainder of the boundary, we will make use of the notation introduced in Section~\ref{sec:mainresults}. Recall that $\partial \Omega \setminus \Gamma$ consists, except for the corners, of finitely many smooth arcs $\Sigma_1, \dots, \Sigma_N$ enumerated following the boundary in positive orientation from one end point of $\Gamma$ to the other, and that we denote the corner points adjacent to each $\Sigma_j$ by $P_j$ and $P_{j-1}$, see Figure~\ref{fig:notationHelp}. Moreover, we define the functions
\begin{align*}
 t_j  \coloneqq (b \cdot \tau)(b \cdot \nu)|_{\Sigma_j}, \quad j = 1, \dots, N;
\end{align*}
note that for the index $j$ for which $\Sigma_j$ equals $\Gamma'$ we have $t_j = 0$ identically as $b$ is normal to $\Gamma'$, and that $t_j$ is constant if $\Sigma_j$ is a straight line segment.

At almost each point of $\partial \Omega \setminus \Gamma$ we can express the vector $b$ as a linear combination of $\tau$ and $\nu$ as in \eqref{eq:bLinearComb} so that on $\partial \Omega \setminus \Gamma$, by the Neumann boundary condition imposed on $u$, \eqref{eq:reint2} becomes
\begin{align*}
b \cdot \nabla u = (b \cdot \tau) \, \partial_\tau u + (b \cdot \nu) \, \partial_\nu u = (b \cdot \tau) \,\partial_\tau u.
\end{align*}
Inserting this expression into the integrand of \eqref{eq:done} and integrating over any arc~$\Sigma_j$, $j = 1, \ldots, N$, by the fundamental theorem of calculus we get
\begin{align*}
 \int_{\Sigma_j} u (b \cdot \nabla u)(b \cdot \nu) \, \textup{d} \sigma
& = \int_{\Sigma_j} (b \cdot \tau) (b \cdot \nu) u \partial_\tau u \, \textup{d} \sigma \\
& = \frac{1}{2} \int_{\Sigma_j} t_j \partial_\tau (u^2) \, \textup{d} \sigma \\
& = \frac{1}{2} \bigg( - \int_{\Sigma_j} u^2 \partial_\tau t_j \, \textup{d} \sigma + t_j (P_j) u^2 (P_j) - t_j (P_{j-1}) u^2 (P_{j-1}) \bigg).
\end{align*}
Using this and $u = 0$ on $\Gamma$, in particular $u (P_0) = u (P_N) = 0$, we obtain
\begin{align}
\label{eq:decomposed}
\begin{split}
 \int_{\partial \Omega} u (b \cdot \nabla u)(b \cdot \nu) \, \textup{d} \sigma
 & = \frac{1}{2} \sum_{j=1}^{N} \bigg( - \int_{\Sigma_j} u^2 \partial_\tau t_j \, \textup{d} \sigma \\
 & \quad + t_j (P_j) u^2 (P_j) - t_j (P_{j-1}) u^2 (P_{j-1}) \bigg) \\
& = \frac{1}{2} \bigg( - \sum_{j=1}^{N}  \int_{\Sigma_j} u^2 \partial_\tau t_j \, \textup{d} \sigma + \sum_{j=1}^{N - 1} [t_j - t_{j + 1}] (P_j) u^2 (P_j) \bigg).
\end{split}
\end{align}
We can now conclude that this boundary integral is non-negative. Indeed, by the assumption \eqref{eq:miraculous} of the theorem, the almost-everywhere defined function $(b \cdot \tau)(b \cdot \nu)$ is non-increasing along $\partial \Omega \setminus \Gamma$, that is
\begin{align*}
 \partial_\tau t_j \leq 0 \quad \text{on}~\Sigma_j, \quad j = 1, \dots, N,
\end{align*}
and $[t_j - t_{j + 1}] (P_j) \geq 0$ at each corner $P_j$, $j = 1, \dots, N - 1$. Hence \eqref{eq:done} follows from \eqref{eq:decomposed}. With \eqref{eq:fundeq} this proves the inequality
\begin{equation*}
\lambda^{\Gamma'}_1 \le \lambda^{\Gamma}_1.
\end{equation*}

Next, we argue that this inequality is actually strict. Assume for a contradiction that equality holds. Then the above considerations imply that both \eqref{eq:curvatureZero} and
\begin{align}
\label{eq:boundIntZero}
 - \sum_{j=1}^{N}  \int_{\Sigma_j} u^2 \partial_\tau t_j \, \textup{d} \sigma + \sum_{j=1}^{N - 1} [t_j - t_{j + 1}] (P_j) u^2 (P_j) = 0
\end{align}
hold. Note that since the function $(b \cdot \tau)(b \cdot\nu)$ is non-increasing, \eqref{eq:boundIntZero} holds if and only if each summand is equal to zero separately. We distinguish three cases.

In the first case, assume that $\Gamma$ is not a straight line segment, i.e., there exists a relatively open subset $\Lambda \subset \Gamma$ on which $\kappa$ is non-zero. In this case we can argue the same way as above to lead \eqref{eq:curvatureZero} to a contradiction.

In the second case, assume that $\Gamma$ is a straight line segment which is not parallel to $\Gamma'$. Let $j$ be such that $\Sigma_j$ is adjacent to $\Gamma'$. Also in this case we can argue analogously to the above reasoning and show that \eqref{eq:normalDerConstant} holds on the segment $\Sigma_j$. The end point shared by $\Gamma$ and $\Sigma_j$, which meet at an angle strictly less than $\pi/2$, is again a critical point of $u$ as $u$ satisfies a Dirichlet boundary condition on $\Gamma$ and a Neumann boundary condition on $\Sigma_j$. This shows that $\partial_\nu u = 0$ on $\Gamma$, which together with Lemma~\ref{lem:continuation principle} yields a contradiction.

Let us finally consider the third case, in which $\Gamma$ is a line segment parallel to $\Gamma'$. In particular, $\Gamma$ and $\Gamma'$ are not adjacent to each other. The equation \eqref{eq:boundIntZero} together with the earlier considerations yields 
\begin{align*}
 [t_j - t_{j + 1}] (P_j) u^2 (P_j) = 0, \qquad j = 1, \dots, N - 1.
\end{align*}
Let now $j$ be such that $P_j$ is the end point of $\Gamma'$. Due to the convexity of $\Omega$ and the requirements on the angles adjacent to $\Gamma$ in the theorem, the interior angle of $\partial \Omega$ at $P_j$ must be strictly larger than $\pi/2$, and, in particular, $t_j (P_j) - t_{j + 1} (P_j) \neq 0$ so that $u(P_j)=0$ must hold. Note that $\lambda_1^{\Gamma'} = \lambda_1^\Gamma$ and \eqref{eq:almost} yield
\begin{align*}
\lambda_1^\Gamma = \lambda_1^{\Gamma'} = \frac{\int_\Omega |\nabla v|^2 \, \textup{d} x}{\int_\Omega |v|^2 \, \textup{d} x},
\end{align*}
which in turn implies that $v$ is an eigenfunction of $- \Delta_{\Gamma'}$ corresponding to $\lambda_1^{\Gamma'}$. By Courant's nodal theorem (see for instance \cite[Theorem 2]{PL56}), $v=b \cdot \nabla u$ must then be either strictly positive or strictly negative inside $\Omega$, which implies that $u$ must be either strictly increasing or strictly decreasing inside $\Omega$ in the direction of the vector $b$. However, the straight line parallel to $b$ through $P_j$ intersects $\partial \Omega$, except for $P_j$, at a point in $\Gamma$, where $u$ satisfies a Dirichlet boundary condition. But the function $u$ vanishes at both its intersection points with $\partial \Omega$, a contradiction. This concludes the proof of Theorem~\ref{thm:mainthm2}.
\end{proof}

\section*{Acknowledgements}
The authors gratefully acknowledge financial support by the grant no.\ 2018-04560 of the Swedish Research Council (VR).


\begin{thebibliography}{99}

\bibitem{A11} M.\,S.\ Agranovich, {\it Mixed problems in a Lipschitz domain for second-order strongly elliptic systems}, Funct.\ Anal.\ Appl.\ 45 (2011), 81--98.

\bibitem{A81} A.\ Azzam, {\it Smoothness properties of solutions of mixed boundary value problems for elliptic equations in sectionally smooth n-dimensional domains}, Ann.\ Polon.\ Math.\ Vol.\ 1, No.\ 40 (1981), 81--93.

\bibitem{B94} R.\ Brown, {\it The mixed problem for Laplace’s equation in a class of Lipschitz domains}, Comm.\ Partial Differential Equations 19 (1994), 1217--1233.
 
\bibitem{EE18} D.E.\ Edmunds and W.\,D.\ Evans, Spectral theory and differential operators, Second Edition, Oxford Univ.\ Press, 2018.

\bibitem{FNO21} V.\ Felli, B.\ Noris, and R.\ Ognibene, {\it Eigenvalues of the Laplacian with moving mixed boundary conditions: the case of disappearing Dirichlet region}, Calc.\ Var.\ Partial Differential Equations 60 (2021), no. 1, Paper No.\ 12, 33 pp.

\bibitem{FNO22} V.\ Felli, B.\ Noris, and R.\ Ognibene, {\it Eigenvalues of the Laplacian with moving mixed boundary conditions: the case of disappearing Neumann region}, J.\ Differential Equations 320 (2022), 247--315.

\bibitem{GWW92} C.\ Gordon, D.\ Webb, and S.\ Wolpert, {\it Isospectral plane domains and surfaces via Riemannian orbifolds}, Invent.\ Math.\ 110 (1992), 1--22.

\bibitem{G85} P.\ Grisvard, Elliptic problems in nonsmooth domains, Monographs and Studies in Mathematics, 24, Pitman (Advanced Publishing Program), Boston, MA, 1985.

\bibitem{KO83} V.\,A.\ Kondrat'ev and O.\,A.\ Oleinik, {\it Boundary value problems for partial differential equations in nonsmooth domains}, Uspekhi Mat.\ Nauk 38 (1983), no.\ 2(230), 3--76.
 
\bibitem{LW86} H.\ Levine and H.\ Weinberger, {\it Inequalities between Dirichlet and Neumann eigenvalues}, Arch.\ Rational Mech.\ Anal.\ 94 (1986), 193--208.

\bibitem{LR17} V.\ Lotoreichik and J.\ Rohleder, {\it Eigenvalue inequalities for the Laplacian with mixed boundary conditions}, J.\ Differential Equations 263 (2017), 491--508.

\bibitem{M00} W.\ McLean, Strongly Elliptic Systems and Boundary Integral Equations, Cambridge Univ.\ Press, 2000.

\bibitem{O06} B.\ O’Neill, Elementary Differential Geometry 2nd ed., Academic Press Inc., 2006.

\bibitem{PL56} \r{A}.\ Pleijel, {\it Remarks on Courant's nodal line theorem}, Comm.\ Pure Appl.\ Math.\ 9 (1956), 543--550.

\bibitem{PR10} A.\ Pressley, Elementary Differential Geometry, Second edition, Springer Undergraduate Mathematics Series. Springer--Verlag London, Ltd., London, 2010.

\bibitem{R14} J.\ Rohleder, {\it Strict inequality of Robin eigenvalues for elliptic differential operators on Lip\-schitz domains}, J.\ Math.\ Anal.\ Appl.\ 418 (2014), 978--984.

\bibitem{R20} J.\ Rohleder, {\it A remark on the order of mixed Dirichlet--Neumann eigenvalues of polygons}, Analysis as a Tool in Mathematical Physics, Birkhäuser, Cham (2020), 570--575.

\bibitem{R21} J.\ Rohleder, {\it Inequalities between Neumann and Dirichlet eigenvalues of Schr\"odinger operators}, J.\ Spectr.\ Theory 11 (2021), 915--933.

\bibitem{S02} R.\ Seeley, {\it Trace expansions for the Zaremba problem}, Comm.\ Partial Differential Equations 27 (2002), 2403--2421.

\bibitem{S15} B.\ Siudeja, {\it Hot spots conjecture for a class of acute triangles}, Math.\ Z.\ 280 (2015), no. 3--4, 783--806.

\bibitem{S16} B.\ Siudeja, {\it On mixed Dirichlet-Neumann eigenvalues of triangles}, Proc.\ Amer.\ Math.\ Soc.\ 144 (2016), 2479--2493.

\bibitem{W70} N.\,M.\ Wigley, {\it Mixed boundary value problems in plane domains with corners}, Math.\ Z.\ 115 (1970), 33--52.

\bibitem{W93} T.\,H.\ Wolff, {\it Recent work on sharp estimates in second-order elliptic unique continuation problems}, J.\ Geom.\ Anal.\ 3 (1993), 621--650.



\end{thebibliography}
\end{document}